\documentclass[11pt]{amsart}
                                                                                
\usepackage{amsmath,amsthm,amsfonts,amssymb, latexsym} 
\newtheorem{thm}{Theorem}[section]

\newtheorem{lem}[thm]{Lemma}

 \newtheorem*{conj}{Conjecture}                                                                               
\newtheorem*{notation}{Notation}
                                                                            
\theoremstyle{definition}
\newtheorem*{defin}{Definition}

\DeclareMathOperator{\dist}{dist}

\newcommand\rr{\mathbb{R}}

\begin{document}

\title[PMT for Lipschitz metrics]{A positive mass theorem for Lipschitz metrics with small singular sets}
\author{Dan A.\ Lee}
\thanks{This work was partially supported by NSF DMS \#0903467 and a PSC CUNY Research Grant.}
\address{CUNY Graduate Center and Queens College}
\email{dan.lee@qc.cuny.edu}
\maketitle

\begin{abstract}
We prove that the positive mass theorem applies to Lipschitz metrics as long as the singular set is low-dimensional, with no other conditions on the singular set.  More precisely, let $g$ be an asymptotically flat Lipschitz metric on a smooth manifold $M^n$, such that $n<8$ or $M$ is spin.  As long as $g$ has bounded $C^2$ norm and nonnegative scalar curvature on the complement of some singular set $S$ of Minkowski dimension less than $n/2$, the mass of $g$ must be nonnegative.  We conjecture that the dimension of $S$ need only be less than $n-1$ for the result to hold.  These results complement and contrast with earlier results of H.~Bray~\cite{Bray:2001}, P.~Miao \cite{Miao:2002}, and Y.~Shi and L.-F.~Tam \cite{Shi-Tam:2002}, where $S$ is a hypersurface.
\end{abstract}

\section{Introduction}
In the study of nonnegative scalar curvature, one would like to formulate some notion of ``weak'' nonnegative scalar curvature for metrics that are not necessarily smooth.   The gold standard for a notion of weak nonnegative scalar curvature would be something like Alexandrov spaces as a weak notion of spaces with lower bounds on sectional curvature.   A good notion of weak nonnegative scalar curvature would be one that implies the same consequences as ``classical'' nonnegative scalar curvature---for example, the positive mass theorem in the asymptotically flat case, or topological restrictions in the compact case. 

An important theorem in this direction was proved by P.~Miao \cite{Miao:2002}, generalizing an earlier result of H.~Bray \cite[Section 6]{Bray:2001}.  See also \cite[Section~3]{Shi-Tam:2002} for the spin case, as well as a more recent proof by D.~McFeron and G.~Sz\'{e}kelyhidi~\cite{McFeron-Szekelyhidi}.  

\begin{thm}\label{hypersurface}
Let $M^n$ be a smooth manifold such that $n<8$ or $M$ is spin.\footnote{The only point of this assumption is to make sure that the classical positive mass theorem is valid on $M$.  If the positive mass theorem is true in all dimensions, then this hypothesis can safely be eliminated}  Let $S$ be a smooth closed hypersurface in $M$, and let $g$ be a complete asymptotically flat metric on $M$ such such that $g$ is $C^2$ \emph{up to} $S$ from each side of it (but not necessarily \emph{across} it) and $C^{2,\alpha}_\mathrm{loc}$ away from $S$.  

Near each point of $S$, $S$ divides $M$ into two sides, which we will call $A$ and $B$.  Let $H_A$ be the mean curvature vector of $S$ as computed by the metric on the $A$ side, and similarly define $H_B$.

If $g$ has nonnegative scalar curvature on the complement of $S$, and at each point of~$S$, $H_A-H_B$ either points toward side $A$ or is zero, then the mass of $g$ is nonnegative in each end.
\end{thm}

Note that the hypotheses of the theorem above require $g$ to be Lipschitz everywhere.  One way to interpret this theorem is that when the singular set of $g$ is a hypersurface $S$ whose induced metric is well-defined regardless of which ``side'' of $S$ one uses to compute it, then the correct notion of weak nonnegative scalar curvature on $S$ is the pointwise mean curvature comparison condition that appears as a hypothesis of the theorem above.  

In this article we consider singular sets of lower dimension and ponder what conditions on $S$ correspond to weak nonnegative scalar curvature.  We find that if $S$ has low enough dimension, then no further conditions are needed.
\begin{thm}\label{maintheorem}
Let $M^n$ be a smooth manifold such that $n<8$ or $M$ is spin.  Let $g$ be a complete asymptotically flat Lipschitz metric on $M$, and let $S$ be a bounded subset whose $n/2$-dimensional lower Minkowski content is zero.
If $g$ has bounded $C^2$-norm and nonnegative scalar curvature on the complement of $S$, then the mass of $g$ is nonnegative in each end.  
\end{thm}
See Section \ref{definitions} for the definition of Minkowski content.  For now, recall that Minkowski content equals Hausdorff measure for well-behaved sets (\textit{e.g.} submanifolds).

There is also a $W^{1,p}$ version of this theorem.
\begin{thm}\label{w1p}
Let $M^n$ be a smooth manifold such that $n<8$ or $M$ is spin.  Let $p>n$, let $g$ be a complete asymptotically flat $W^{1,p}_{\mathrm{loc}}$ (and hence continuous) metric on $M$, and let $S$ be a bounded subset whose $\frac{n}{2}(1-\frac{n}{p})$-dimensional lower Minkowski content is zero.  If $g$ has bounded $C^2$-norm and nonnegative scalar curvature on the complement of $S$, then the mass of $g$ is nonnegative in each end.  
\end{thm}

It might seem surprising that one does not have to place any other conditions on the behavior of $g$ at singular set, but as we will see in the proof, the Lipschitz (or $W^{1,p}$) condition is very restrictive.  Essentially, $g$ is too regular for the scalar curvature to be truly singular on a small set.  We use the same technique as in \cite{Miao:2002}.  Note that if $S$ is a closed submanifold, the proofs of Theorems \ref{maintheorem} and \ref{w1p} are much simpler.

The dimensional restriction of $n/2$ seems to an unnecessary artifact of the conformal method used in the proof.  Also note that Theorems \ref{maintheorem} and \ref{w1p} do not include rigidity results for Euclidean space.
\begin{conj}
Let $M^n$ be a smooth manifold such that $n<8$ or $M$ is spin.  Let $g$ be a complete asymptotically flat Lipschitz metric on $M$, and let $S$ be a bounded subset whose $(n-1)$-dimensional lower Minkowski content is zero.
If $g$ has bounded $C^2$-norm and nonnegative scalar curvature on the complement of $S$, then the mass of $g$ is nonnegative in each end.  Moreover, if the mass of any end is zero, then $(M,g)$ must be isometric to Euclidean space.
\end{conj}
One might try to prove the spin case of this conjecture using a spinor argument, following \cite{Shi-Tam:2002}.  Or one might try to prove the conjecture using Ricci flow as in \cite{McFeron-Szekelyhidi}.  One advantage of the Ricci flow method is that it is more likely to produce a rigidity result.

\section{Definitions}\label{definitions}
\begin{defin}
Let $g$ be a continuous Riemannian metric on a smooth manifold $M^n$ where $n\geq3$.  Then $(M,g)$ is an \emph{asymptotically flat manifold} if and only if there is a compact set $K\subset M$ such that $M\smallsetminus K$ is a disjoint union of ends, $E_\ell$, such that
each end is diffeomorphic to $\rr^n$ minus a ball, and in each of
these coordinate charts, the metric $g_{ij}$ is $C^2$ and satisfies 
\begin{align*}
g_{ij}&=\delta_{ij}+O(|x|^{-\sigma})\\ 
g_{ij,k}&=O(|x|^{-\sigma-1})\\
g_{ij,kl}&=O(|x|^{-\sigma-2})\\
 R_g&=O(|x|^{-\tau}), 
\end{align*} 
for some $\sigma>(n-2)/2$ and $\tau>n$, where the commas denote partial derivatives in the coordinate chart, and $R_g$ denotes the scalar curvature of~$g$.

We define the \emph{mass} of each end $E_\ell$ by the formula
\[m(E_\ell,g)={1\over
2(n-1)\omega_{n-1}}\lim_{\rho\to\infty}\int_{S_\rho}
\sum_{i,j=1}^n(g_{ij,i}-g_{ii,j})\nu_j d\mu,\]
 where $\omega_{n-1}$ is the area of the standard unit
$(n-1)$-sphere, $S_{\rho}$ is the coordinate sphere in $E_\ell$ of radius
$\rho$, $\nu$ is its outward unit normal, and $d\mu$ is the Euclidean area
element on $S_{\rho}$.  The mass is well-defined on each end of an asymptotically flat manifold.
\end{defin}

\begin{defin}
For a subset $S$ of a Riemannian manifold$(M^n, g)$, the \emph{$m$-dimensional lower Minkowski content} of $S$ is 
\[ \liminf_{\epsilon\to0} \frac{\mathcal{L}_g^n(S_\epsilon)}{\alpha_{n-m}\epsilon^{n-m}}\]
where $\mathcal{L}_g^n$ is Lebesgue measure with respect to $g$, $S_\epsilon$ is the $\epsilon$-neighborhood of $S$, and $\alpha_{n-m}$ is the volume of the unit ball in $\rr^{n-m}$.
\end{defin}
The $m$-dimensional lower Minkowski content provides an upper bound (up to constant) for $m$-dimensional Hausdorff measure, and they are the same for rectifiable sets (see \cite[Chapter 3.2]{Federer-book} for details).  In particular, the condition of zero Minkowski content in Theorems \ref{maintheorem} and \ref{w1p} is only slightly stronger than the condition of zero Hausdorff measure.

\section{Proof of Theorem \ref{maintheorem}}\label{main}

First, we briefly sketch out the proof, which is straightforward.  Choose $M^n$, $g$, and $S$ as in the statement of Theorem \ref{maintheorem}.  We mollify the metric $g$ to get a smooth metric $g_\epsilon$ in such a way that $g_\epsilon=g$ outside of the $2\epsilon$-neighborhood $S_{2\epsilon}$.  The precise smoothing of $g$ does not matter much.  The only important property of the smoothing is that, using the hypotheses on~$g$, we have that $g_\epsilon$, $g_\epsilon^{-1}$, and $\partial g_\epsilon$ are bounded independently of $\epsilon$, while $\partial\partial g_\epsilon=O(\epsilon^{-1})$, with respect to a particular atlas.  By the formula for the scalar curvature of $g_\epsilon$, it follows that $R_{g_\epsilon}=O(\epsilon^{-1})$.  The hypothesis about Minkowski content tells us (roughly) that the volume of $S_\epsilon$ is $o(\epsilon^{n/2})$.  Thus
\begin{equation}\label{goal}
\int_{S_{2\epsilon}} |R_{g_\epsilon}|^{n/2}\,dg=o(1).
\end{equation}
From there, a standard argument (as in \cite{Miao:2002}) tells us that we can conformally deform $g_\epsilon$ to have nonnegative scalar curvature, without changing the mass too much.  Applying the classical positive mass theorem to the new, smooth manifold of nonnegative scalar curvature, we find that the original manifold $(M,g)$ has mass greater than a small negative number that is $o(1)$ in $\epsilon$.  Taking the limit as $\epsilon$ approaches zero, the result follows.  In what follows, we describe an explicit smoothing that yields \eqref{goal}.

We choose a finite atlas $U_1,\ldots,U_N$ for $M$.  By asymptotic flatness and continuity of $g$, we can choose these $U_k$ so that $g$ is uniformly equivalent to the background Euclidean metric of each patch.  That is, on each coordinate patch, we have 
\[ C^{-1}\delta_{ij}\leq g_{ij}\leq C\delta_{ij} \]
as positive definite symmetric bilinear forms.  
We choose a partition of unity  $\psi_1,\ldots, \psi_N$ subordinate to this cover. 
On each patch $U_k$, we will define a smoothing $g^k_\epsilon$ of $g$ that is defined on the support of $\psi_k$, which we denote $U'_k$.   We then obtain a smoothing $g_\epsilon$ of $g$ by defining $g_\epsilon=\sum_{k=1}^N \psi_k g^k_\epsilon$.  

\begin{notation}In what follows, we will use a generic constant $C$ to mean some large number that may depend on $(M,g)$ and the choices of $U_k$ and $\psi_k$.  The only thing that will be important to us is that $C$ is independent of $\epsilon$.
\end{notation}
Given a coordinate patch $U_k$, we wish to define $g^k_\epsilon$.  Let $\varphi$ be a nonnegative smooth function supported on the unit ball in $\rr^n$ whose integral is $1$.  The standard way to smooth $g$ is to convolve it with 
\[\varphi_\epsilon(x) := \epsilon^{-n}\varphi(x/\epsilon).\]
  However, we want to smooth $g$ in such a way that it does not change $g$ away from a neighborhood of $S$.  In order to do that, we need the following simple lemma.  

\begin{lem}\label{sigma}
For each $\epsilon>0$, on each coordinate patch $U_k$, there exists a nonnegative smooth function $\sigma$  such that $\sigma=\epsilon$ on the Euclidean neighborhood $S_\epsilon$ of $S$ in $U_k$, and $\sigma=0$ outside $S_{2\epsilon}$, while $|\partial\sigma|\le 3$ and $|\partial\partial\sigma|\le C\epsilon^{-1}$ everywhere.
\end{lem}
\begin{proof}
Define a continuous function 
\[s(x)=\left\{\begin{array}{ll}
\epsilon & \text{for }x\in S_{4\epsilon/3}\\
5\epsilon - 3\dist(x,S) & \text{for }x\in S_{5\epsilon/3}\smallsetminus S_{4\epsilon/3}\\
0 & \text{for }x\notin S_{5\epsilon/3}
\end{array}\right.\]
Then we can define 
\[\sigma(x)=\int_{\rr^n} s(x-y)\varphi_{\epsilon/6} (y)\,dy.\]
Clearly, $\sigma$ is a nonnegative smooth function such that $\sigma=\epsilon$ on $S_\epsilon$ and $\sigma=0$ outside $S_{2\epsilon}$.  We just need to check the bounds on derivatives.
\begin{align*}
|\sigma(x_1)-\sigma(x_2)| & \leq\int_{\rr^n} |s(x_1-y)-s(x_2-y)|\varphi_{\epsilon/6} (y)\,dy \\
&\leq \int_{\rr^n} 3 |x_1-x_2| \varphi_{\epsilon/6} (y)\,dy\\
&=3|x_1-x_2|,
\end{align*}
where we used the Lipschitz property of $s$ in the second line.  Thus $|\partial\sigma|\le 3$.

We know that 
\begin{align*}
 \partial\sigma(x)&= \int_{\rr^n} s(x-y) \partial\varphi_{\epsilon/6}(y)\,dy \\ 
&=\int_{\rr^n} s(x-y)  \left(\frac{6}{\epsilon}\right)^{n+1}  \partial\varphi\left(\frac{6y}{\epsilon}\right)\,dy 
\end{align*}
Arguing as above, 
\begin{align*}
|\partial\sigma(x_1)-\partial\sigma(x_2)| 
&\leq \int_\rr^n 3|x_1-x_2|    \left(\frac{6}{\epsilon}\right)^{n+1}    \left|\partial\varphi \left(\frac{6y}{\epsilon}\right)\right|\,dy \\
& = \frac{18}{\epsilon} |x_1-x_2| \int_{\rr^n} |\partial\varphi (y)|\,dy \\
&=\frac{C}{\epsilon}|x_1-x_2|.
\end{align*}
 Thus $|\partial\partial\sigma|\le C\epsilon^{-1}$.
\end{proof}

For small enough $\epsilon$, we define $g^k_\epsilon$ on the patch $U'_k$ by the formula
\[ (g^k_\epsilon)_{ij}(x)=\int_{\rr^n}g_{ij}(x-\sigma(x)y)\varphi(y)\,dy 
=\int_{\rr^n}g_{ij}(y)\varphi_{\sigma(x)}(x-y)\,dy.\]
Keep in mind that the function $\sigma$ described by the lemma above depends on the patch $U_k$, the singular set $S$, and $\epsilon$.  Clearly, each component of $g^k_\epsilon$ is smooth.  We now claim that $|\partial g^k_\epsilon|\le C$ and $|\partial\partial g^k_\epsilon|\le C\epsilon^{-1}$.  For ease of notation, let us prove these inequalities for each component, individually.  The lemma below proves the claim.

\begin{lem}\label{convolve}
Let $\sigma$ be the function described in Lemma \ref{sigma}, and let $f$ be a Lipschitz function on $U_k$ that has bounded $C^2$-norm on the complement of $S$.  For small enough $\epsilon$, if we define the function
\[ f_\epsilon(x)=\int_{\rr^n}f(x-\sigma(x)y)\varphi(y)\,dy 
=\int_{\rr^n}f(y)\varphi_{\sigma(x)}(x-y)\,dy\]
on the set $U'_k$, then $|\partial f_\epsilon|\le C$ and $|\partial\partial f_\epsilon|\le C\epsilon^{-1}$, where $C$ may depend on the supremum and Lipschitz constant of $f$.
\end{lem}
\begin{proof}
On the complement of $S_\epsilon$, the result follows easily from differentiating the first formula for $f_\epsilon$ above and using the $C^2$ bound on $f$ and the bounds on $|\partial\sigma|$ and $|\partial\partial\sigma|$ from Lemma \ref{sigma}.  So we need only consider the region $S_\epsilon$.  But in this region we have $\sigma(x)=\epsilon$ by construction, and therefore
\[ f_\epsilon(x)
=\int_{\rr^n}f(y)\varphi_\epsilon(x-y)\,dy 
\]
is just the usual mollification formula.  Then since $f$ is Lipschitz, a standard computation shows that $|\partial f_\epsilon|$ is bounded.  Moreover, for $x\in S_\epsilon$,
\begin{align*}
\partial f_\epsilon (x)
&=\int_{\rr^n} f(y) \partial\varphi_\epsilon(x-y)\,dy\\
&=\int_{\rr^n} f(y) \epsilon^{-n-1}\partial\varphi\left(\frac{x-y}{\epsilon}\right)\,dy\\
&=\int_{\rr^n} f(x-\epsilon y) \epsilon^{-1}\partial\varphi(y)\,dy.
\end{align*}
So for any $x_1,x_2\in S_\epsilon$, 
\begin{align*}
|\partial f_\epsilon(x_1)-\partial f_\epsilon(x_2)|
&=\left| \int_{\rr^n} [f(x_1-\epsilon y)-f(x_2-\epsilon y)] \epsilon^{-1}\partial\varphi(y)\,dy\right|\\
&\le  \int_{\rr^n} |f(x_1-\epsilon y)-f(x_2-\epsilon y)| \epsilon^{-1}|\partial\varphi(y)|\,dy\\
&\le  \int_{\rr^n}C|x_1-x_2| \epsilon^{-1}|\partial\varphi(y)|\,dy\\
&\le \frac{C}{\epsilon}|x_1-x_2|.
\end{align*}
The result follows.
\end{proof}

Setting $g_\epsilon=\sum_{k=1}^N \psi_k g^k_\epsilon$, Lemma \ref{convolve} implies that in each coordinate chart~$U_k$, 
$|\partial (g_\epsilon)_{ij}|\le C$ and $|\partial\partial (g_\epsilon)_{ij}|\le C\epsilon^{-1}$ for some $C$ independent of~$\epsilon$.  From looking at how scalar curvature depends on the metric, it is clear that $|R_{g_\epsilon}|\le C\epsilon^{-1}$ for some $C$.  Meanwhile, $g=g_\epsilon$ outside $S_{2\epsilon}$, and by our assumption on the Minkowski content of $S$, we have 
\begin{align*}
\int_{S_{2\epsilon}} |R_{g_\epsilon}|^{n/2}\,dg &\le \mathcal{L}_g^n(S_{2\epsilon})\sup |R_{g_\epsilon}|^{n/2} \\
&= o(\epsilon^{n/2}) O(\epsilon^{-n/2})\\
&=o(1),
\end{align*}
which is our desired estimate \eqref{goal}.  The rest of the proof of Theorem \ref{maintheorem} proceeds exactly as in \cite{Miao:2002}.  (Technically, since we are using \emph{lower} Minkowski content, it is inaccurate to say that $\mathcal{L}_g^n(S_{2\epsilon})= o(\epsilon^{n/2})$, but the argument still works since we only need to use a subsequence of $\epsilon$'s approaching zero.  Also, we were careless about the distinction between defining $S_\epsilon$ using the metric $g$ versus the Euclidean metric on each chart, but by uniform equivalent of metrics, this sloppiness is inconsequential.)

\section{Proof of Theorem \ref{w1p}}

The proof of the $W^{1,p}$ version of Theorem \ref{maintheorem} requires only slight modification.  Choose $M^n$, $g$, $S$, and $p$ as in the statement of Theorem \ref{maintheorem}.  First, observe that because of the $C^2$ bounds on $g$, we can see that $|R_{g_\epsilon}|=O(\epsilon^{-1})$ on the complement of $S_\epsilon$, just as in the Lipschitz case, and we now have even better bounds on $\mathcal{L}^n_g(S_{2\epsilon})$, so we have
\[\int_{S_{2\epsilon}\smallsetminus S_\epsilon} |R_{g_\epsilon}|^{n/2}\,dg=o(1).\] 
  Therefore, in order to establish \eqref{goal}, it is sufficient to show that
\begin{equation}\label{newgoal}
\int_{S_\epsilon} |R_{g_\epsilon}|^{n/2}\,dg=o(1).
\end{equation}

Next we use a $W^{1,p}$ version of Lemma \ref{convolve}.
\begin{lem}\label{convolve-w1p}
Let $\sigma$ be the function described in Lemma \ref{sigma}, and let $f\in W^{1,p}_{\mathrm{loc}}(U_k)$ such that $f$ has bounded $C^2$-norm on the complement of $S$.  For small enough $\epsilon$, if we define the function
\[ f_\epsilon(x)=\int_{\rr^n}f(x-\sigma(x) y)\varphi(y)\,dy 
=\int_{\rr^n}f(y)\varphi_{\sigma(x)}(x-y)\,dy\]
on the set $U'_k$, then $\|\partial f_\epsilon\|_{L^p(S_\epsilon\cap U'_k)} \le C$ and 
$|\partial\partial f_\epsilon|\le C\epsilon^{-1-\frac{n}{p}}$, where $C$ may depend on the $W^{1,p}$ norm of $f$.
\end{lem}
\begin{proof}
Recall that for $x\in S_\epsilon$, $\sigma(x)=\epsilon$, so that the formula for
$f_\epsilon$ is the usual mollification formula.  A standard argument using H\"{o}lder's inequality shows that
\[ \|\partial f_\epsilon\|_{L^p(S_\epsilon \cap U'_k)} \le \|\partial f\|_{L^p(S_{2\epsilon}\cap U_k)}\le C.\]
For any $x\in S_\epsilon$, if $q$ is chosen so that $\frac{1}{p}+\frac{1}{q}=1$, then
\begin{align*}
|\partial\partial f_\epsilon(x)|
&= \left|\partial\partial \int_{\rr^n} f(x-y)\varphi_\epsilon(y)\,dy\right|\\
&= \left|\partial\int_{\rr^n} \partial f(x-y)\varphi_\epsilon(y)\,dy\right|\\
&= \left|\partial\int_{\rr^n} \partial f(y)\varphi_\epsilon(x-y)\,dy\right|\\
&= \left|\int_{S_{2\epsilon}\cap U_k} \partial f(y) \partial\varphi_\epsilon(x-y)\,dy\right|\\
&\le \|\partial f\|_{L^p(S_{2\epsilon}\cap U_k)}  \left(\int_{B_\epsilon(x)} |\partial\varphi_\epsilon(x-y)|^q\,dy\right)^{\frac{1}{q}}\\
&\le C\left(\epsilon^{n} (\epsilon^{-n-1})^q\right)^{\frac{1}{q}}\\
&= C\epsilon^{-1-\frac{n}{p}}.
\end{align*}
\end{proof}
Since the scalar curvature is a contraction of $\partial\partial g + g^{-1}*g^{-1}*\partial g *\partial g$, on any of the coordinate patches, we can use Lemma \ref{convolve-w1p} to estimate
\begin{align*}
 \int_{S_\epsilon\cap U_k} |R_{g_\epsilon}|^{n/2}\,dg
 &\le  C\int_{S_\epsilon\cap U_k} |\partial\partial g_\epsilon|^{n/2}\,dx +  C\int_{S_\epsilon\cap U_k} |\partial g_\epsilon|^{n}\,dx \\
  &\le  \mathcal{L}^n(S_\epsilon\cap U_k) \left(C\epsilon^{-1-\frac{n}{p}}\right)^{n/2} + o(1)\\
  &= o\left(\epsilon^{\frac{n}{2}\left(1+\frac{n}{p}\right)}\right) \epsilon^{-\frac{n}{2}\left(1+\frac{n}{p}\right)}+o(1)\\
  &=o(1).
  \end{align*}
 As in Section \ref{main}, there is some justifiable carelessness in the computation above.  And once again, the result of the proof follows exactly as in \cite{Miao:2002}.

\bibliographystyle{hamsplain}
\bibliography{research2011}

\end{document}